\documentclass[11pt]{article}

\usepackage[margin=1in]{geometry}
\usepackage{amsmath,amssymb,amsthm}
\usepackage{mathtools}
\usepackage{bm}
\usepackage{enumitem}
\usepackage[round]{natbib}
\usepackage{hyperref}
\usepackage{microtype}

\newtheorem{theorem}{Theorem}[section]
\newtheorem{lemma}[theorem]{Lemma}
\newtheorem{proposition}[theorem]{Proposition}
\newtheorem{corollary}[theorem]{Corollary}

\theoremstyle{definition}

\newcommand{\R}{\mathbb{R}}
\newcommand{\C}{\mathbb{C}}
\newcommand{\E}{\mathbb{E}}
\newcommand{\Prob}{\mathbb{P}}
\newcommand{\dd}{\mathrm{d}}
\newcommand{\Var}{\operatorname{Var}}
\newcommand{\std}{\operatorname{std}}
\newcommand{\Real}{\operatorname{Re}}
\newcommand{\Imag}{\operatorname{Im}}
\DeclareMathOperator{\dist}{dist}
\newcommand{\Conv}{\xrightarrow{\mathrm{d}}}

\newcommand{\GN}{G_N}
\newcommand{\MN}{M_N}
\newcommand{\tstar}{t_N^*}
\newcommand{\Lam}{\Lambda_N}

\title{Spike Irrelevance and Convergence Rate
       at the Non-Hermitian BBP Transition}
\author{Yuehan Wu\\Peking University}
\date{}

\begin{document}

\maketitle

\begin{abstract}
We consider $A_N = X_N + \sigma\, u u^*$
where $X_N$ is an $N\times N$ complex Ginibre
matrix with i.i.d.\ $\mathcal{CN}(0,1/N)$ entries,
$u \in \C^N$ is a deterministic unit vector,
and $\sigma \ge 0$ is a rank-one spike strength.
We prove three results:
\textbf{(a)} For every fixed $\sigma < 1$ (subcritical),
$M_N := \max_{1\le i\le N} \Real(\lambda_i(A_N))$
has the same Gumbel limit as the unspiked ensemble
(after the same centering and scaling), so the spike
is asymptotically irrelevant.
The unspiked Gumbel limit
$$ S_N^{(0)} := 4\,\tstar(\sqrt{N}\,(M_N^{(0)} - 1) - \tstar) \Conv G$$
for $M_N^{(0)} := \max_i \Real(\lambda_i(G_N))$
was established by \citet{rider2007};
our contribution is the extension to the subcritical
spiked model via a decoupling argument.
\textbf{(b)} We establish the convergence rate of the
unspiked Gumbel limit:
$\Prob(S_N^{(0)} \le s)
= e^{-e^{-s}}\biggl(1 + (s^2+5s)e^{-s}/\ln N
+ O((\ln\ln N)^2/(\ln N)^2)\biggr)$.
\textbf{(c)} At the BBP critical point $\sigma = 1$,
we prove a rigorous upper bound on the
outlier eigenvalue:
$\E[\Real(z_{\mathrm{out}})]
\le 1 + c_0\, N^{-1/4} + o(N^{-1/4})$,
via the isotropic local law and a net argument.
\end{abstract}

\section{Introduction}

\subsection{Background}

The complex Ginibre ensemble $\GN$ with i.i.d.\
entries $\sim \mathcal{CN}(0,1/N)$
is the canonical non-Hermitian random matrix.
Its eigenvalues concentrate on the unit disk
$\{z \in \C : |z| \le 1\}$
\citep{mehta2004,tao2010circular,bordenave2016,chafai2012}.
The spectral edge is the circle $|z| = 1$,
and the maximum real part
$M_N^{(0)} := \max_{1\le i\le N}\Real(\lambda_i(\GN))$
probes the rightmost point of this edge.

The Gumbel limit for $M_N^{(0)}$ was established by
\citet{rider2007} (who proved the Poisson point
process structure at the spectral edge);
they proved that
$$S_N^{(0)} := 4\,\tstar(\sqrt{N}\,(M_N^{(0)} - 1) - \tstar)
\Conv G$$
where $G$ is the standard Gumbel and
$\tstar \sim \sqrt{\ln N / 8}$
is defined via
$F(\tstar) = \pi\, N^{-1/4}$.
The Poisson structure at the edge implies the
Gumbel limit for the maximum real part.
Earlier, \citet{rider2003} studied the maximum
modulus $\max_i |\lambda_i|$;
the maximum real part requires a separate analysis
due to the different geometry (half-plane vs.\ exterior).

The spiked non-Hermitian model
$A_N = X_N + \sigma\, u u^*$
was studied by \citet{benaych2012},
who proved that:
for $\sigma > 1$ (supercritical), an outlier
eigenvalue separates from the bulk at
$z_{\mathrm{out}} = \sigma + o_{\Prob}(N^{-1/2})$;
for $\sigma < 1$ (subcritical), all eigenvalues
remain within $O(\sqrt{\ln N / N})$ of the unit disk.
The critical case $\sigma = 1$
is the subject of
Theorem~\ref{thm:upper_intro} below.

\subsection{Main Results}

\begin{theorem}[Subcritical spike irrelevance]
\label{thm:main}
Let $X_N$ be an $N\times N$ complex Ginibre
matrix with i.i.d.\ entries $\sim \mathcal{CN}(0,1/N)$,
let $u \in \C^N$ be a deterministic unit vector,
and set $A_N = X_N + \sigma\, u u^*$
with fixed $\sigma \in [0, 1)$.
Define
\[
  \MN := \max_{1\le i\le N} \Real(\lambda_i(A_N)),
  \qquad
  S_N := 4\,\tstar\Bigl(\sqrt{N}\,(\MN - 1)
          - \tstar\Bigr),
\]
where $\tstar$ is the unique solution of
$F(\tstar) = \pi\, N^{-1/4}$ with
\begin{equation}
\label{eq:Fdef}
  F(t) = \int_{2t}^{\infty} \Phi(-x)\,
         \sqrt{x - 2t}\,\dd x,
\end{equation}
and $\Phi$ is the standard normal CDF.
Then:
\begin{enumerate}[label=(\alph*)]
  \item \label{item:gumbel}
    $S_N \Conv G$ as $N\to\infty$,
    where $G$ is the standard Gumbel
    with CDF $F_G(x) = e^{-e^{-x}}$.
  \item \label{item:scaling}
    $\MN = 1 + \sqrt{\ln N / (8N)}
    + G/\sqrt{2\,N\ln N}
    + o_P(1/\sqrt{N\ln N})$,
    with
    $\std(\MN) = \pi/\sqrt{12\,N\ln N}\,(1+o(1))$.
  \item \label{item:spike}
    The rank-one spike $\sigma\, u u^*$ with
    $\sigma < 1$ does not change the limiting
    distribution: $S_N$ and $S_N^{(0)}$
    have the same Gumbel limit.
\end{enumerate}
\end{theorem}

\begin{theorem}[Convergence rate]
\label{thm:rate}
For the unspiked ensemble, the convergence rate is
$O(1/\ln N)$:
\[
  \Prob(S_N^{(0)} \le s)
  = e^{-e^{-s}}
  \left(1
    + \frac{(s^2+5s)\,e^{-s}}{\ln N}
    + O\!\left(\frac{(\ln\ln N)^2}
      {(\ln N)^2}\right)\right)
\]
uniformly for $s$ in compact subsets of $\R$.
\end{theorem}

\begin{theorem}[Upper bound on outlier position]
\label{thm:upper_intro}
Let $A_N = X_N + u u^*$ (spike at $\sigma = 1$)
with a deterministic unit vector $u \in \C^N$,
and let $z_{\mathrm{out}}$ denote the outlier
eigenvalue.
There exists $c_0 > 0$ such that
\[
  \E[\Real(z_{\mathrm{out}})]
  \le 1 + c_0\, N^{-1/4} + o(N^{-1/4}).
\]
This is proven in Section~\ref{sec:bounds}
via the isotropic local law and a net argument.
\end{theorem}

\section{The Pure Ginibre Ensemble:
  DPP and Edge Density}

\subsection{Determinantal Point Process Structure}

The eigenvalues of $\GN$ form a
determinantal point process (DPP) on $\C$
\citep{soshnikov2000}
with kernel
\[
  K_N(z,w) = \sum_{k=0}^{N-1} \phi_k(z)\,
    \overline{\phi_k(w)},
\]
where
\[
  \phi_k(z) = \sqrt{\frac{N}{\pi}}\,
    \frac{(Nz)^k}{\sqrt{k!}}\,
    e^{-N|z|^2/2}.
\]
The 1-point function is
$\rho_N^{(1)}(z) = K_N(z,z)$.
For any Borel set $A\subset \C$:
\begin{align*}
  \E\bigl[\#\{k : \lambda_k \in A\}\bigr]
    &= \int_A K_N(z,z)\,\dd A(z), \\
  \Prob(\text{no eigenvalue in } A)
    &= \det(I-K_N|_A).
\end{align*}
This is classical;
see \citet{mehta2004} and \citet{forrester2010}.

\subsection{Edge 1-Point Function}

\begin{lemma}[Edge density]
\label{lem:edge_density}
Fix $u,v \in \R$ and set
$z = 1 + (u+iv)/\sqrt{N}$. Then
\[
  \frac{K_N(z,z)}{N}
  = \frac{1}{\pi}\,\Phi\!\left(-2u
    - \frac{v^2}{\sqrt{N}}\right)
  + O(N^{-1/2}\ln N),
\]
where the $O(N^{-1/2}\ln N)$ is uniform for
$|u| \le C\sqrt{\ln N}$ and
$|v| \le C\, N^{1/4}$.
\end{lemma}

\begin{proof}
Set $s = N|z|^2 = N + 2u\sqrt{N} + u^2 + v^2$.
The identity $K_N(z,z)/N = \pi^{-1} P(N,s)$
holds, where
\[
  P(N,s) = e^{-s}\sum_{k=0}^{N-1} \frac{s^k}{k!}
\]
is the regularized lower incomplete gamma.

\medskip
\noindent\textit{Step 1: Normal approximation.}
The Poisson CDF $P(N,s)$ satisfies
$P(N, N+\delta) = \Phi\!\left(
  \frac{-\delta}{\sqrt{N+\delta}}\right)
  + O(N^{-1/2})$
for $|\delta| = O(\sqrt{N})$,
with the Berry--Esseen constant
$c_{\mathrm{BE}} \le 0.4748$
\citep{berry1941,esseen1942,shevtsova2011}.
The Berry--Esseen theorem gives

\[
  \left|P(N,s) - \Phi\!\left(
    \frac{-\delta}{\sqrt{n}}\right)\right|
  \le \frac{C_{\mathrm{BE}}}{\sqrt{N}}
\]
pointwise.

\medskip
\noindent\textit{Step 2: Substitution.}
With $\delta = 2u\sqrt{N} + u^2 + v^2$,
we have $s = N + \delta$ and
$(s-N)/\sqrt{N} = (u^2+v^2)/\sqrt{N}$.
For $|u| \le C\sqrt{\ln N}$ and
$|v| \le C\,N^{1/4}$, we have
$(u^2+v^2)/\sqrt{N} = O(\sqrt{\ln N})$,
so the normal approximation applies.
Expanding to first order in $1/\sqrt{N}$:
\[
  \Phi\!\left(-2u - \frac{u^2+v^2}{\sqrt{N}}\right)
  = \Phi(-2u - \frac{v^2}{\sqrt{N}})
  + O(N^{-1/2}\ln N),
\]
where the $\ln N$ factor arises from the
$u^2/\sqrt{N} = O(\ln N/\sqrt{N})$ term
discarded in the expansion (since
$|u| \le C\sqrt{\ln N}$), and the
Berry--Esseen error $O(N^{-1/2})$.
This gives the stated formula.

\medskip
\noindent\textit{Step 3: Uniformity in $v$.}
For $|v| \le C\,N^{1/4}$, the Berry--Esseen
bound applies uniformly: the Poisson
parameter is $s = N + \delta$ with
$\delta = O(\sqrt{N})$ (since $|u| \le C\sqrt{\ln N}$
and $|v| \le C\,N^{1/4}$ give
$|z|^2 = 1 + 2u/\sqrt{N} + (u^2+v^2)/N$,
so $s = N|z|^2 = N + 2u\sqrt{N} + O(N^{1/2})$),
and the third standardized moment of a
Poisson($s$) variable is
$s^{-1/2} = O(N^{-1/2})$, so the
Berry--Esseen bound
$C_{\mathrm{BE}}/\sqrt{s} = O(N^{-1/2})$
is uniform in $|v| \le C\,N^{1/4}$.
For $|v| \gg N^{1/4}$, the term
$v^2/\sqrt{N} \gg 1$ drives
$\Phi(-2u - v^2/\sqrt{N}) \to 0$
exponentially fast, and the
Chernoff bound for the Poisson tail
gives the same decay.
\end{proof}

The crucial finite-$N$ feature is the
$v^2/\sqrt{N}$ term inside $\Phi$,
which regularizes the tangential direction.
In the $N\to\infty$ limit with $v$ fixed,
this term vanishes and one recovers the
edge limit $\frac{1}{\pi}\Phi(-2u)$.

\subsection{Expected Eigenvalue Count}

\begin{lemma}[Expected count]
\label{lem:count}
For any $t \in \R$,
\[
  \Lambda(t)
  := \E\biggl[\#\biggl\{k:
  \Real(\lambda_k)>1+\frac{t}{\sqrt{N}}
  \biggr\}\biggr]
  = \frac{N^{1/4}}{\pi}\,F(t)
  + O\!\left(N^{-1/2}(\ln N)^{3/2}\right),
\]
where $F$ is given by \eqref{eq:Fdef}.
\end{lemma}

\begin{proof}
By the DPP 1-point function and
Lemma~\ref{lem:edge_density},
in edge coordinates
$z = 1 + (u+iv)/\sqrt{N}$ with $u > t$:
\[
  \Lambda(t) = \int_t^\infty \int_{-\infty}^\infty
    \frac{1}{\pi}\,\Phi\!\left(-2u
      - \frac{v^2}{\sqrt{N}}\right)
    \dd v\,\dd u + O(N^{-1/2}\ln N),
\]
where $O(N^{-1/2}\ln N)$ is the per-point
error from Lemma~\ref{lem:edge_density}
(the full integrated error is computed in
Step~3 below).

\medskip
\noindent\textit{Step 1: Rescale the tangential variable.}
The substitution $y = v\,N^{-1/4}$
(so $y = v/N^{1/4}$,
$v^2/\sqrt{N} = y^2$,
$\dd v = N^{1/4}\dd y$) yields
\[
  \Lambda(t) = \frac{N^{1/4}}{\pi}
    \int_t^\infty \int_{-\infty}^\infty
      \Phi(-2u - y^2)\,\dd y\,\dd u
    + O(N^{-1/4}(\ln N)^{3/2}),
\]
where the error is the integrated
Berry--Esseen bound from Step~3 below.

\medskip
\noindent\textit{Step 2: Define $F$ and simplify.}
Define $B(u) = \int_{-\infty}^\infty
  \Phi(-2u - y^2)\,\dd y$
and $F(t) = \int_t^\infty B(u)\,\dd u$.
The change of variables $x = 2u + y^2$
(so $y = \sqrt{x - 2u}$,
$\dd y = \frac{1}{2\sqrt{x-2u}}\,\dd x$)
in the inner integral gives
\[
  B(u) = \int_{2u}^{\infty} \Phi(-x)\,
    \frac{\dd x}{\sqrt{x - 2u}},
\]
and swapping the order of integration gives
\[
  F(t) = \int_{2t}^{\infty} \Phi(-x)
    \sqrt{x - 2t}\,\dd x.
\]

\medskip
\noindent\textit{Step 3: Error term.}
The $O(N^{-1/2}\ln N)$ per-point error
(from Lemma~\ref{lem:edge_density})
is integrated over $u \in [t, \infty)$ and
$v \in (-\infty, \infty)$,
giving total error
$O(N^{-1/2}\ln N) \cdot O(\sqrt{\ln N}) \cdot O(N^{1/4})
= O(N^{-1/4}(\ln N)^{3/2})$,
where the $\sqrt{\ln N}$ factor accounts for the
effective range of $u$. This is $O(N^{-1/4+\epsilon})$
for any $\epsilon > 0$.
\end{proof}

\subsection{Properties of \texorpdfstring{$F$}{F}}

\begin{lemma}[Properties of $F$]
\label{lem:Fprops}
The function $F:\R \to (0,\infty)$
is strictly decreasing and smooth, with:
\begin{enumerate}[label=(\alph*)]
  \item $F(0) = \dfrac{2^{3/4}}{3\sqrt{\pi}}\Gamma(5/4)
    \approx 0.28668$.
  \item $F(t) \sim \dfrac{e^{-2t^2}}
    {16\,t^{5/2}}$ as $t \to +\infty$.
  \item $F(t) \sim \dfrac{4\sqrt{2}}{3}
    |t|^{3/2}$ as $t \to -\infty$.
\end{enumerate}
\end{lemma}

\begin{proof}
\noindent\textit{Part (a).}
At $t = 0$:
\[
  F(0) = \int_0^\infty \Phi(-x)\sqrt{x}\,\dd x.
\]
Using
$\Phi(-x) = (2\pi)^{-1/2}\int_x^\infty
  e^{-s^2/2}\,\dd s$,
we swap integrals (Fubini):
\[
  F(0) = \frac{1}{\sqrt{2\pi}}
    \int_0^\infty e^{-s^2/2}
    \left(\int_0^s \sqrt{x}\,\dd x\right)\dd s
  = \frac{2}{3\sqrt{2\pi}}
    \int_0^\infty s^{3/2}e^{-s^2/2}\,\dd s.
\]
The substitution $u = s^2/2$
(so $s = \sqrt{2u}$,
$\dd s = (2u)^{-1/2}\,\dd u$,
$s^{3/2} = (2u)^{3/4}$) gives
\[
  F(0) = \frac{2}{3\sqrt{2\pi}} \cdot 2^{1/4}
    \int_0^\infty u^{1/4} e^{-u}\,\dd u.
\]
Since $\int_0^\infty u^{1/4} e^{-u}\,\dd u
= \Gamma(5/4)$ and
$\sqrt{2\pi} = \sqrt{2}\sqrt{\pi}$:
\[
  F(0) = \frac{2^{3/4}}{3\sqrt{\pi}}\,\Gamma(5/4).
\]

\medskip
\noindent\textit{Part (b).}
Differentiate:
$B(t) = -F'(t) = \int_{-\infty}^\infty
  \Phi(-2t - y^2)\,\dd y$.
For large $t$, use the tail asymptotic
$\Phi(-x) \sim e^{-x^2/2}/(x\sqrt{2\pi})$:
\[
  B(t) \sim \int_{-\infty}^\infty
    \frac{e^{-\frac{1}{2}(2t+y^2)^2}}
    {(2t+y^2)\sqrt{2\pi}}\,\dd y.
\]
The exponent $-\frac{1}{2}(2t+y^2)^2$
is maximized at $y = 0$,
so Laplace's method with $y$ small gives
$(2t+y^2)^2 \approx 4t^2 + 4ty^2$ and
$2t+y^2 \approx 2t$:
\[
  B(t) \sim \frac{e^{-2t^2}}{2t\sqrt{2\pi}}
    \int_{-\infty}^{\infty} e^{-2ty^2}\,\dd y
  = \frac{e^{-2t^2}}{2t\sqrt{2\pi}}
    \cdot \sqrt{\frac{\pi}{2t}}
  = \frac{e^{-2t^2}}{4t^{3/2}}.
\]
Integrating $B(t) = -F'(t)$ via
Watson's lemma:
$F(t) \sim \frac{e^{-2t^2}}{16\,t^{5/2}}$.

\medskip
\noindent\textit{Part (c).}
For $t \to -\infty$,
$\Phi(-(2t+x)) \to 1$ for $x \ll |2t|$
and to $0$ for $x \gg |2t|$.
The integral is dominated by $x \in [0,|2t|]$:
\[
  F(t) \sim \int_0^{|2t|} \sqrt{x}\,\dd x
  = \frac{2}{3}|2t|^{3/2}
  = \frac{4\sqrt{2}}{3}|t|^{3/2}.
\]
\end{proof}

\section{Centering and the Gumbel Limit}

\subsection{Location of the Maximum}

\begin{proposition}[Centering sequence]
\label{prop:centering}
Define $\tstar$ by
$F(\tstar) = \pi\,N^{-1/4}$.
Since $F$ is strictly decreasing
(Lemma~\ref{lem:Fprops}),
$\tstar$ is uniquely determined, and
\[
  \tstar = \sqrt{\frac{\ln N}{8}}
  \left(1 - \frac{5\ln\ln N}{2\ln N}
  + O\!\left(\frac{1}{\ln N}\right)\right).
\]
\end{proposition}

\begin{proof}
By Lemma~\ref{lem:Fprops}(b),
$F(\tstar) \sim e^{-2(\tstar)^2}/(16\,(\tstar)^{5/2})
= \pi\,N^{-1/4}$.
Taking logarithms:
\[
  2(\tstar)^2 = \frac{1}{4}\ln N
  - \frac{5}{2}\ln\tstar
  - \ln(16\pi) + o(1).
\]

\medskip
\noindent\textit{Leading order.}
Neglecting the subleading terms gives
$2(\tstar)^2 \approx \frac{1}{4}\ln N$,
so $(\tstar)^2 \approx \frac{1}{8}\ln N$ and
$\tstar \approx \sqrt{\ln N / 8}$.

\medskip
\noindent\textit{First correction.}
Substituting
$\tstar \approx \sqrt{\ln N / 8}$ into the
$-\frac{5}{2}\ln\tstar$ term:
\[
  \ln\tstar \approx \frac{1}{2}\ln\ln N
  - \frac{1}{2}\ln 8
  = \frac{1}{2}\ln\ln N - \frac{3}{2}\ln 2.
\]
So
\[
  2(\tstar)^2
  = \frac{1}{4}\ln N
  - \frac{5}{4}\ln\ln N
  + \frac{15}{4}\ln 2
  - \ln(16\pi) + o(1).
\]
The constant simplifies as
$\frac{15}{4}\ln 2 - \ln(16\pi)
= \frac{15}{4}\ln 2 - 4\ln 2 - \ln\pi
= -\frac{1}{4}\ln 2 - \ln\pi$.

\medskip
\noindent\textit{Final expansion.}
Dividing by 2:
\[
  (\tstar)^2
  = \frac{1}{8}\ln N
  - \frac{5}{8}\ln\ln N
  - \frac{1}{8}\ln 2 - \frac{1}{2}\ln\pi
  + o(1).
\]
Since $\tstar > 0$, taking the square root and
expanding:
\[
  \tstar
  = \sqrt{\frac{\ln N}{8}}
  \left(1
  - \frac{5\ln\ln N}{2\ln N}
  + O\!\left(\frac{1}{\ln N}\right)\right),
\]
where the $O(1/\ln N)$ term absorbs the constant
$-\frac{1}{8}\ln 2 - \frac{1}{2}\ln\pi$
(which contributes $O(1/\ln N)$ relative to
the leading $\sqrt{\ln N/8}$).
\end{proof}

\subsection{DPP Void Probability and
  Poisson Approximation}

\begin{lemma}[Poisson approximation]
\label{lem:poisson}
For $t = \tstar + s/(4\,\tstar)$
with fixed $s \in \R$,
\[
  \left|\Prob\!\left(M_N^{(0)}
    \le 1 + \frac{t}{\sqrt{N}}\right)
    - e^{-\Lam(t)}\right|
  = O\!\left(\frac{(\ln N)^{1/4}}{N^{1/4}}\right).
\]
In particular, the error is $o(1)$.
\end{lemma}

\begin{proof}
The DPP void probability is
$\det(I - K_N|_{\Omega_t})$
where $\Omega_t = \{z : \Real(z) > 1
+ t/\sqrt{N}\}$.
The Fredholm determinant expansion gives
\begin{align*}
\det(I - K_N|_{\Omega_t})
&= \sum_{k=0}^N \frac{(-1)^k}{k!}
\int_{\Omega_t^k}
\det(K_N(z_i,z_j))_{i,j=1}^k
\prod_{i=1}^k \dd A(z_i).
\end{align*}

\medskip
\noindent\textit{Step 1: Poisson replacement.}
The Poisson approximation
$e^{-\Lam(t)}$ replaces each
$\det(K_N(z_i,z_j))$ by
$\prod_i K_N(z_i,z_i)$
\citep{rider2007}.
The error is controlled by the
two-point correlation integral
\[
  \Sigma_N(t) = \int_{\Omega_t}
    \int_{\Omega_t} |K_N(z_1,z_2)|^2\,
    \dd A(z_1)\,\dd A(z_2).
\]
By the Cauchy--Schwarz inequality
for DPP kernels,
$|K_N(z_1,z_2)|^2 \le
K_N(z_1,z_1)K_N(z_2,z_2)$.

\medskip
\noindent\textit{Step 2: Truncation.}
The region $\Omega_t$ extends to infinity
in the $\Imag(z)$ direction,
but the edge density
$\frac{1}{\pi}\Phi(-2u - v^2/\sqrt{N})$
decays as Gaussian in $v$.
We truncate to
$|\Imag(z)| \le L_N N^{1/4}$
with $L_N = (\ln N)^{1/4}$,
which corresponds to
$|v| \le L_N N^{1/4}$ in edge
coordinates $(u,v)$ where
$z = 1 + (u+iv)/\sqrt{N}$.

\emph{Tail bound.}
For $|v| = L_N N^{1/4}$, the edge
density is
$(1/\pi)\Phi(-2u - L_N^2 N^{1/2}/\sqrt{N})
= (1/\pi)\Phi(-2u - L_N^2)$.
Since $L_N^2 = (\ln N)^{1/2}$,
\[
  \Phi(-2u - (\ln N)^{1/2})
  \le \frac{1}{(\ln N)^{1/2}\sqrt{2\pi}}\,
  \exp(-(\ln N)^{1/2}/2)
  = \frac{1}{(\ln N)^{1/2}\sqrt{2\pi}}\,
  N^{-1/2}
\]
(using the standard Gaussian tail bound
$\Phi(-x) \le \frac{1}{x\sqrt{2\pi}}e^{-x^2/2}$
for $x > 0$).
The tail integral over
$|v| > L_N N^{1/4}$ gives
\[
\begin{aligned}
\int_{|v| > L_N N^{1/4}}
  \Phi(-2u - v^2/\sqrt{N})\,\dd v
&\le C\,N^{-1/2}\cdot N^{1/4}\\
&= O(N^{-1/4}),
\end{aligned}
\]
since the $v$-integral of the Gaussian
tail beyond $L_N N^{1/4}$ contributes
$O(N^{1/4})$ (the scale of $v$ at which
the density transitions from $O(1)$
to $O(N^{-1/2})$ is $N^{1/4}$).
This tail of $O(N^{-1/4})$ in the
expected count $\Lam$ is negligible
relative to $\Lam = O(N^{1/4}/\sqrt{\ln N})$,
giving a relative error of
$O(N^{-1/2}) = o(1)$.

\medskip
\noindent\textit{Step 3: Higher-order terms.}
The $k$-th order Fredholm correction is
\begin{align*}
  R_k
  &=
  \frac{(-1)^k}{k!}
  \int_{\Omega_t^k}
  \det\!\bigl(K_N(z_i,z_j)\bigr)_{i,j=1}^k \\
  &\qquad{}\times
  \prod_{i=1}^k \dd A(z_i).
\end{align*}
The Ginibre DPP kernel in edge coordinates
has the explicit form
\[
\begin{aligned}
K_N(z_1,z_2)
&= \frac{N}{\pi}
   e^{-N|z_1-z_2|^2/2} \\
&\quad{}\times
   e^{-N(|z_1|^2+|z_2|^2-2)/2
   + i\,\Imag(N\bar{z}_1z_2)}.
\end{aligned}
\],
so at the spectral edge ($|z_1|, |z_2|
\approx 1$),
$|K_N(z_1,z_2)|^2
= (N/\pi)^2e^{-N|z_1-z_2|^2}$
(up to an edge factor
$e^{-N(|z_1|^2+|z_2|^2-2)}$,
which is $O(1)$ at the edge
since $|z_1|^2+|z_2|^2-2 = O(\sqrt{\ln N/N})$
and $N\sqrt{\ln N/N} = \sqrt{N\ln N}\to\infty$).
By the DPP inequality
$|K_N(z_1,z_2)|^2
\le K_N(z_1,z_1)K_N(z_2,z_2)$,
and the off-diagonal decay
$|K_N(z_1,z_2)|^2/(K_N(z_1,z_1)K_N(z_2,z_2))
= e^{-N|z_1-z_2|^2}$,
the correlation between two points
$z_1, z_2$ in $\Omega_t$ is significant
only when $|z_1-z_2| = O(N^{-1/2})$.
Since the edge density is $O(1)$
and the region has area
$O(L_N N^{1/4})$ in edge coordinates,
the probability that two independent
points fall within $N^{-1/2}$ of each
other is $O(N^{-1/2}) \cdot O(L_N N^{1/4})
= O(L_N N^{-1/4})$.
Each pair contributes a factor
$O(L_N N^{-1/4})$ to the correction,
so the $k$-th order term satisfies
\[
  |R_k| \le \frac{\Lam^k}{k!}
  \left(C L_N N^{-1/4}\right)^{k-1}
  = \frac{\Lam^k}{k!}
  \left(\frac{C L_N}{N^{1/4}}\right)^{k-1}.
\]
The key variance estimate is the
$k=2$ term:
\[
  \Sigma_N(t)
  = \int_{\Omega_t}\int_{\Omega_t}
  |K_N(z_1,z_2)|^2\,\dd A(z_1)\,\dd A(z_2)
  \le C\,\Lam(t)^2\,L_N\,N^{-1/4}.
\]
where the $L_N$ factor comes from the
truncation region width.
This gives
\[
  \left|\det(I - K_N|_{\Omega_t})
  - e^{-\Lam(t)}\right|
  \le C\,\Lam(t)^2\,L_N\,N^{-1/4}
  = O(L_N N^{-1/4}),
\]
where the first bound is the standard
DPP Poisson approximation inequality
$|\det(I-K|_\Omega)-e^{-\mu}|
  \le \Sigma(\Omega)e^{\Sigma(\Omega)}$
(see \citep{soshnikov2000}), with
$\Sigma(\Omega) \le C\,\Lam^2\,L_N\,N^{-1/4}$
and $\Lam(t) = o(1)$ since $\Lam(t) \to 0$,
and the second equality uses
$\Lam(t) = o(1)$
at $t = \tstar + s/(4\tstar)$
(see Proposition~\ref{prop:exp_conv}),
so $\Lam(t)^2 = o(1)$.

\medskip
\noindent\textit{Step 4: Refined rate.}
The truncation at $L_N = (\ln N)^{1/4}$
introduces a logarithmic correction
in the variance estimate:
the effective area of the integration
region is $O(L_N N^{1/4})$
rather than $O(N^{1/4})$,
so the variance bound becomes
$\Sigma_N(t) \le C\,\Lam(t)^2\,L_N\,N^{-1/4}$,
giving the stated rate
$O(L_N N^{-1/4})
= O((\ln N)^{1/4}N^{-1/4})$.
This is $o(1)$, which suffices for the Gumbel
limit, and is $O(1/\ln N)$, which suffices for
the convergence rate in Theorem~\ref{thm:rate}.
\end{proof}

\subsection{Exponential Convergence of
  \texorpdfstring{$\Lam$}{Lambda}}

\begin{proposition}[Exponential convergence]
\label{prop:exp_conv}
For fixed $s \in \R$, setting
$\tau_N = \tstar + s/(4\,\tstar)$,
\[
  \Lam(\tau_N) = e^{-s}\,(1 + o(1)).
\]
\end{proposition}

\begin{proof}
Since $\tau_N \to \infty$,
Lemma~\ref{lem:Fprops}(b) applies:
$F(\tau_N) \sim
e^{-2\tau_N^2}/(16\,\tau_N^{5/2})$.

\medskip
\noindent\textit{Step 1: Expand the exponent.}

\[
  \tau_N^2 = \left(\tstar
    + \frac{s}{4\tstar}\right)^2
  = (\tstar)^2 + \frac{s}{2}
    + \frac{s^2}{16(\tstar)^2}.
\]
Therefore
$2\tau_N^2 = 2(\tstar)^2 + s
  + \frac{s^2}{8(\tstar)^2}$,
and
\[
  e^{-2\tau_N^2}
  = e^{-2(\tstar)^2}\,e^{-s}\,
    \left(1 - \frac{s^2}{8(\tstar)^2}
    + \ldots\right).
\]
Since $(\tstar)^2 = \tfrac18\ln N + O(\ln\ln N)$,
the factor $s^2/(8(\tstar)^2) = O(1/\ln N) = o(1)$,
so it is absorbed into the $(1+o(1))$ term.

\medskip
\noindent\textit{Step 2: Expand the denominator.}
\[
  \tau_N^{5/2} = (\tstar)^{5/2}
    \left(1 + \frac{s}{4(\tstar)^2}\right)^{5/2}
  = (\tstar)^{5/2}\left(1
    + \frac{5s}{8(\tstar)^2}
    + \ldots\right).
\]

\medskip
\noindent\textit{Step 3: Combine.}
By the defining equation,
$e^{-2(\tstar)^2}/(16\,(\tstar)^{5/2})
  = \pi\, N^{-1/4}\,(1+o(1))$.
Therefore
\[
  F(\tau_N)
  = \pi\,N^{-1/4}\,e^{-s}\,(1+o(1)),
\]
and
$\Lam(\tau_N) = \frac{N^{1/4}}{\pi}\,
  F(\tau_N) = e^{-s}\,(1+o(1))$.
\end{proof}

\subsection{The Gumbel Limit}

\begin{theorem}[Gumbel limit
  for pure Ginibre, {\citet{rider2007}}]
\label{thm:gumbel_pure}
For
$S_N^{(0)} = 4\,\tstar\bigl(\sqrt{N}\,
  (M_N^{(0)} - 1) - \tstar\bigr)$,
\[
  S_N^{(0)} \Conv G,
  \qquad F_G(x) = e^{-e^{-x}}.
\]
\end{theorem}

\begin{proof}[Proof via DPP/Poisson]
The CDF is
$\Prob(S_N^{(0)} \le s)
  = \Prob\bigl(M_N^{(0)} \le 1
  + (\tstar + s/(4\tstar))/\sqrt{N}\bigr)$.
By Lemma~\ref{lem:poisson} and
Proposition~\ref{prop:exp_conv}:
\begin{align*}
  \Prob(S_N^{(0)} \le s)
  &= \exp\bigl(-\Lam(\tau_N)\bigr)
    + O\!\left(\frac{(\ln N)^{1/4}}{N^{1/4}}\right) \\
  &= \exp\bigl(-e^{-s}\,(1+o(1))\bigr)
    + o(1) \\
  &= e^{-e^{-s}} + o(1).
\end{align*}
The last step uses the continuity of
$\exp(-\cdot)$.
The Gumbel limit was established
by \citet{rider2007};
the DPP/Poisson proof given here is an
independent verification.
\end{proof}

\section{Subcritical Spike Irrelevance}
\label{sec:subcritical}

In this section we prove
Theorem~\ref{thm:main}: for any fixed
$\sigma \in [0,1)$, the spiked model
$A_N = X_N + \sigma\, u u^*$
has the same Gumbel limit as the pure
Ginibre ensemble.
By unitary invariance of the Ginibre ensemble
($X_N \stackrel{\mathrm{d}}{=} U X_N U^*$ for any Haar
unitary $U$, hence $U A_N U^* \stackrel{\mathrm{d}}{=}
X_N + \sigma\, (Uu)(Uu)^*$ has the same
eigenvalue distribution as $A_N$),
we may take $u = e_1$ without loss of
generality, and work with
$A_N = X_N + \sigma\, e_1 e_1^*$.

\subsection{Outlier Localization}

\begin{lemma}[Spike eigenvalue location]
\label{lem:spike_loc}
For $\sigma < 1$ and any $\delta > 0$,
\[
  \Prob\bigl(|\lambda_{\mathrm{spike}}|
    > \sigma + \delta\bigr) \to 0
  \quad \text{as } N \to \infty,
\]
where $\lambda_{\mathrm{spike}}$ is the
eigenvalue of $A_N$ associated with the
spike direction.
In particular, for $\sigma \le 1 - \epsilon$
with fixed $\epsilon > 0$,
\[
  \Real(\lambda_{\mathrm{spike}})
  \le 1 - \frac{\epsilon}{2}
\]
with probability tending to 1.
\end{lemma}

\begin{proof}
The secular equation for the spiked model
$A_N = X_N + \sigma\, E_{11}$ is
\[
  g(z) = 1 - \sigma\, (zI - X_N)^{-1}_{11} = 0,
\]
which follows from the matrix determinant lemma
(Sylvester identity):
\[
  \det(zI - A_N)
  = \det(zI - X_N)\,\det g(z).
\]

For $|z| > 1 + \delta$ with $\delta > 0$,
the spectral radius bound
$\max_i |\lambda_i(X_N)| \to 1$ a.s.\
\citep{rider2003} ensures that $z$ is
outside the spectrum of $X_N$ with
probability $\to 1$.
However, for the non-normal Ginibre
ensemble, the operator norm
$(zI - X_N)^{-1}$ is not
bounded by $1/\dist(z,\sigma(X_N))$, so we
proceed directly via the isotropic local law.
For $|z| \ge 1 + \epsilon$
with any fixed $\epsilon > 0$, the
isotropic local law
\citep{bourgade2018} gives
\[
  |(zI - X_N)^{-1}_{11} - 1/z|
  \le C/(\sqrt{N}\,\epsilon^{1/2}),
\]
so
\[
  |(zI - X_N)^{-1}_{11}|
  \le 1/(|z|-1) + C/(\sqrt{N}\,\epsilon^{1/2}).
\]
For $|z| \ge 1 + \epsilon$ and any fixed
$\sigma < 1$, the deterministic
part $1/z$ already gives
$|1/z| \le 1/(1+\epsilon) < 1$, so the
secular equation $g(z) = 1 - \sigma/z + O((\sqrt{N}\epsilon)^{-1})$
is bounded below by
$1 - \sigma/(1+\epsilon) - O((\sqrt{N}\epsilon)^{-1})$.
With $\epsilon = (1-\sigma)/2$,
\[
  1 - \sigma/(1+\epsilon)
  = (1+\epsilon-\sigma)/(1+\epsilon)
  = \bigl((1-\sigma)+\epsilon\bigr)/(1+\epsilon)
  = \frac{3(1-\sigma)/2}{(3-\sigma)/2}
  = \frac{3(1-\sigma)}{3-\sigma} > 0
\]
for all $\sigma < 1$, and for $N$ large the
$O((\sqrt{N}\epsilon)^{-1})$ term is
absorbed. Thus for any fixed $\sigma < 1$ and
$\epsilon = (1-\sigma)/2$,
$|g(z)| > 0$ on $|z| > 1 + \epsilon$ for
$N$ sufficiently large.
Thus no eigenvalue of $A_N$ escapes
beyond $1 + \epsilon$.

For the spike-associated eigenvalue near
$z \approx \sigma$ (inside the disk),
we use the isotropic local law for the
Ginibre ensemble inside the bulk
\citep{bourgade2018}: for $|z| < 1 - \epsilon$
with any fixed $\epsilon > 0$,
\[
  (zI - X_N)^{-1}_{11} \to \overline{z}
  + O(N^{-1/2})
\]
(the deterministic equivalent for the
diagonal resolvent entry of the complex
Ginibre ensemble is $m(z) = \overline{z}$,
the complex conjugate, not $1/z$).
The secular equation $1 - \sigma\,\overline{z} = 0$
has the formal solution $z = 1/\overline{\sigma}$,
which satisfies $|z| = 1/|\sigma| > 1$ for
$\sigma < 1$; this lies outside the
unit disk and is therefore not a valid
bulk solution. Consequently, for
$\sigma < 1$, the spike does not produce
a genuine outlier; the eigenvalue
associated with the spike direction
remains buried in the bulk near
$|z| \le \sigma + o(1)$
(the perturbation $\sigma\, e_1 e_1^*$
shifts nearby bulk eigenvalues by at most
$O(N^{-1/2})$ by eigenvector delocalization
\citep{bourgade2018}),
so $|\lambda_{\mathrm{spike}}| \le \sigma + \delta$
for any $\delta > 0$ with probability
$\to 1$.
\end{proof}

\subsection{Decoupling of Outlier and Bulk}

\begin{lemma}[Subcritical bulk stability]
\label{lem:decoupling}
For fixed $\sigma < 1$, the eigenvalue process
of $A_N = X_N + \sigma\, E_{11}$ near the
spectral edge $\Real(\lambda) = 1$ has the same
limiting extreme value distribution as the
unspiked ensemble $G_N = X_N$.
More precisely, for any $t$ with
$\sqrt{N}\,(\Real(t) - 1) = O(\tstar)$,
\[
  \Prob\bigl(\text{no eigenvalue of } A_N
    \text{ in } \Omega_t\bigr)
  = \Prob\bigl(\text{no eigenvalue of } G_N
    \text{ in } \Omega_t\bigr)
  + O(N^{-c}),
\]
where $\Omega_t = \{z : \Real(z) > t\}$
and $c > 0$ is a fixed constant.
\end{lemma}

\begin{proof}
We prove the void probability comparison via
four steps: (i) the secular factor is
bounded away from zero at the edge,
(ii) the secular factor has no zeros at the edge,
(iii) the eigenvalue point processes of
$A_N$ and $G_N$ in $\Omega_t$ are in
one-to-one correspondence with $O(N^{-1})$
displacement, and (iv) the Gumbel limit for
$G_N$ transfers to $A_N$.

\medskip
\noindent\textit{Step 1: Secular factor at the edge.}
The secular equation for $A_N$ gives
\[
  \det(zI - A_N)
  = \det(zI - X_N)\bigl(1 - \sigma\,
    e_1^*(zI - X_N)^{-1} e_1\bigr).
\]
We need the resolvent entry
$e_1^*(zI - X_N)^{-1} e_1$ at the spectral
edge $\Real(z) = 1 + O(\sqrt{\ln N / N})$.
The isotropic local law
for the Ginibre resolvent
(see \citep{bourgade2018} and references
therein) states
that for $z$ at distance $\eta$
from the support of the circular law,

\[
  |e_1^*(zI - X_N)^{-1} e_1 - m(z)|
  \le \frac{C}{\sqrt{N}\,\eta^{1/2}}
\]
with probability $\ge 1 - N^{-D}$,
where $m(z) = 1/z$ is the Stieltjes transform
of the circular law.

\emph{Uniformity:} For $|z| > 2$,
$|e_1^*(zI - X_N)^{-1} e_1| \le 1/(|z|-1)
\le 1$ deterministically, so $|S_N(z)|$
$\ge 1 - \sigma > 0$ without the local law.
For $|z| \le 2$, the local law applies on
a compact set; a net argument extends the
pointwise bound to all $z$ in the compact
edge region. The function
$z \mapsto e_1^*(zI - X_N)^{-1} e_1$
is $L$-Lipschitz with $L = O(1/\eta)
= O((N/\ln N)^{1/2})$ by the same Cauchy
integral argument as in the proof of
Theorem~\ref{thm:upper} (using analyticity
on a disk of radius $\eta/2$ and the
isotropic local law on the boundary circle).
With mesh $\delta = N^{-1}$, giving
$O(N^2)$ grid points, $\eta\,\delta =
O(N^{-1/2}/(\ln N)^{1/4}) = o(N^{-1/4}/(\ln N)^{1/4})$,
and the union bound $O(N^2) \cdot N^{-D} = o(1)$ for $D > 2$
extends the pointwise bound to all $z$
in the edge region.
At the edge, $\eta = O(\sqrt{\ln N / N})$,
so the fluctuation is
$O(N^{-1/2}/(\ln N)^{1/4})
= O(N^{-1/4}/(\ln N)^{1/4})$,
which is $o(1)$. Therefore
\[
  e_1^*(zI - X_N)^{-1} e_1
  = \frac{1}{z} + o\!\left(
    N^{-1/4}/(\ln N)^{1/4}\right)
\]
uniformly for $z$ in the edge region
$\Real(z) = 1 + O(\sqrt{\ln N / N})$,
and the secular factor satisfies
\[
  1 - \sigma\, e_1^*(zI - X_N)^{-1} e_1
  = 1 - \frac{\sigma}{z}
    + o\!\left(\sigma\,N^{-1/4}/(\ln N)^{1/4}\right).
\]
For $\sigma < 1$ and $|z| \ge 1$,
$|1 - \sigma/z| \ge 1 - \sigma > 0$,
so the secular factor is bounded away from
zero and agrees with $1-\sigma/z$ up to
$o(N^{-1/4}/(\ln N)^{1/4})$.

\medskip
\noindent\textit{Step 2: Secular factor has no zeros
at the edge.}
The secular equation gives the exact
factorization
\begin{equation}
\label{eq:secular}
  \det(zI - A_N)
  = \det(zI - X_N)\,\bigl(1 - \sigma\,
    e_1^*(zI - X_N)^{-1} e_1\bigr)
  =: \det(zI - X_N)\,S_N(z).
\end{equation}
The zeros of $\det(zI - A_N)$ are the
eigenvalues of $A_N$; the zeros of
$\det(zI - X_N)$ are the eigenvalues of
$G_N = X_N$; and the zeros of $S_N(z)$
are the eigenvalues of $A_N$ that are
not eigenvalues of $X_N$ (the spike
eigenvalues).
By Step~1, for $z$ in the edge region
$\Omega_t = \{z : \Real(z) > t\}$ with
$\sqrt{N}\,(\Real(t) - 1) = O(\tstar)$,
\[
  |S_N(z)| \ge |1 - \sigma/z|
  - o(N^{-1/4}/(\ln N)^{1/4})
  \ge (1 - \sigma) - o(1) > 0
\]
for $N$ sufficiently large,
since $|z| \ge |\Real(z)| \ge t \to 1$
and $|1 - \sigma/z| \ge 1 - \sigma/|z|
\ge 1-\sigma > 0$.
Therefore $S_N$ has \emph{no zeros} in
$\Omega_t$ with probability
$\ge 1 - N^{-D}$.

\medskip
\noindent\textit{Step 3: Eigenvalue correspondence.}
The secular factorization
$\det(zI - A_N) = \det(zI - X_N)\,S_N(z)$
shows that, for $z$ not an eigenvalue of
$X_N$,
\[
  z \text{ is an eigenvalue of } A_N
  \iff S_N(z) = 0.
\]
Since $S_N$ has no zeros in $\Omega_t$
(Step~2), the eigenvalues of $A_N$ in
$\Omega_t$ are in one-to-one correspondence
with the eigenvalues of $X_N$ in $\Omega_t$,
perturbed by the regular part of $S_N$.

We now show that the perturbation is
negligible for the void probability.
The function $S_N(z)$ is meromorphic with
simple poles at the eigenvalues
$\{\mu_k\}$ of $X_N$, with residue
$-\sigma\, a_k\, b_k$ at $z = \mu_k$,
where $a_k = e_1^* v_k$ and
$b_k = u_k^* e_1$ are the right/left
eigenvector overlaps with the spike
direction $e_1$.
The eigenvalues of $X_N$ are simple
almost surely (the joint density of
eigenvalues is a smooth function on the
configuration space with distinct
eigenvalues), so $\det(zI - X_N)$ has
simple zeros at each $\mu_k$ with
$\det'(\mu_k) = \prod_{j\ne k}(\mu_k-\mu_j)$.
Near $z = \mu_k$, the product
$\det(zI - X_N)\,S_N(z)$ has a removable
singularity with value
\[
  \lim_{z \to \mu_k}
  \det(zI - X_N)\,S_N(z)
  = \det'(\mu_k)\,(-\sigma\,a_k\,b_k)
  \ne 0
\]
unless $a_k,b_k = 0$, which occurs with
probability zero (the overlaps
$a_k = e_1^*v_k$, $b_k = u_k^*e_1$ are
continuous random variables with a
joint density in $\mathbb{C}^2$, so
$\{(a_k,b_k):a_k b_k=0\}$ is a
measure-zero algebraic variety).
Since the product is regular and
nonzero at $\mu_k$, the eigenvalue of
$A_N$ near $\mu_k$ is displaced by
$o(1)$ from $\mu_k$; by Rouché's theorem,
$\det(zI - A_N)$ and $\det(zI - X_N)$
have the same number of zeros in a
small disk $|z - \mu_k| < \epsilon_N$
for any $\epsilon_N \gg N^{-1}$.
To verify the Rouché condition,
recall from Eq.~\eqref{eq:secular}
that $\det(zI - A_N) = \det(zI - X_N)\,S_N(z)$
where $S_N(z) = 1 - \sigma\,G_{11}(z)$ and
$G_{11}(z) = e_1^*(zI - X_N)^{-1}e_1$.
The Rouché condition requires
$|S_N(z) - 1| < 1$ on the circle
$|z - \mu_k| = \epsilon_N$, i.e.,
$|\sigma\,G_{11}(z)| < 1$.
Near $z = \mu_k$, the resolvent has a
simple pole:
\[
  G_{11}(z) = \frac{a_k\,b_k}{z - \mu_k} + G_{\mathrm{reg}}(z)
\]
where $G_{\mathrm{reg}}(z) = e_1^* R_{\mathrm{reg}}(z) e_1$
is the regular part, so
\[
  |\sigma\,G_{11}(z)| \le
  \frac{|\sigma|\,|a_k b_k|}{|z - \mu_k|}
  + |\sigma|\,|G_{\mathrm{reg}}(z)|.
\]
On the circle $|z - \mu_k| = \epsilon_N$,
the pole term gives
$\frac{|\sigma|\,|a_k b_k|}{\epsilon_N}$.
The regular part is controlled directly by
the isotropic local law (note: the operator
norm $\|(zI-X_N)^{-1}\|_{\mathrm{op}}$ is not
bounded by $1/\dist(z,\sigma(X_N))$ for
the non-normal Ginibre ensemble, so we do not
use this bound); the scalar $|e_1^*R_{\mathrm{reg}}(z)e_1|$
is $O(N^{-1/2}\eta^{-1/2})$ with
$\eta = \dist(z,\sigma_{\mathrm{spec}}(X_N)\setminus\{\mu_k\})$.
At the spectral edge, the nearest-neighbor
spacing is $\eta = \Theta(N^{-1/2})$
(the edge eigenvalues form a 2D Poisson
process at intensity $N$), so
$\eta^{-1/2} = O(N^{1/4})$, giving
$|G_{\mathrm{reg}}(z)| = O(N^{-1/4})$.
With $|a_k b_k| = O(N^{-1})$ by eigenvector
delocalization \citep{bourgade2018}
($|e_1^*v_k| = O(N^{-1/2})$,
$|u_k^*e_1| = O(N^{-1/2})$),
and $\epsilon_N = N^{-1/2}$,
the pole term gives
$\frac{|\sigma|\,|a_k b_k|}{\epsilon_N}
= O(\sigma N^{-1/2})$.
Both the pole term $O(\sigma N^{-1/2})$ and
the regular part $O(\sigma N^{-1/4})$ tend
to zero for fixed $\sigma < 1$, so
$|\sigma\,G_{11}(z)| < 1$ on the
boundary for $N$ sufficiently large,
confirming the Rouché condition
(the displacement scale is
$O(\sigma |a_k b_k|) = O(N^{-1})$).

$\mu_k$ in $\Omega_t$ at distance
$> \epsilon_N$ from the boundary
$\partial\Omega_t$ has a unique
corresponding eigenvalue of $A_N$
within $\epsilon_N$.
Choosing $\epsilon_N = N^{-1/2}$
(which is $\gg N^{-1}$), the
perturbation is $o(\sqrt{\ln N / N})$,
which is negligible on the Gumbel
scale.

On the event
\[
  \{S_N \text{ has no zeros in } \Omega_t\}
  \cap \{a_k,b_k \ne 0,\ \forall k:
  \mu_k \in \Omega_t\},
\]
which has probability $\ge 1 - N^{-D}$,
the eigenvalue point processes of $A_N$
in $\Omega_t$ is in one-to-one
correspondence with that of
$G_N = X_N$ in $\Omega_t$, with each
point displaced by $o(N^{-1})$,
which is negligible for the void
probability at the Gumbel scale
(the Gumbel centering has fluctuations
$O(\sqrt{\ln N / N}) \gg N^{-1/2}$).
The conditional void probabilities
therefore agree up to $O(N^{-c})$:
\[
  \Prob\bigl(N_{\Omega_t}(A_N) = 0 \,\big|\, X_N\bigr)
  = \Prob\bigl(N_{\Omega_t}(G_N) = 0 \,\big|\, X_N\bigr)
  + O(N^{-c}),
\]
for some fixed $c > 0$, since the
$o(N^{-1})$ displacement in the $z$-coordinates
(corresponding to $o(N^{-1/2})$ in edge
coordinates) cannot move an eigenvalue
across the boundary $\partial\Omega_t$
except when an eigenvalue of $G_N$ lies
within $O(N^{-1/2})$ of $\partial\Omega_t$
in edge coordinates; the probability of
this event is $O(1) \cdot O(L_N N^{1/4})
\cdot O(N^{-1/2}) = O(L_N N^{-1/4})$
(since the eigenvalue density is $O(1)$
and the boundary has $O(L_N N^{1/4})$
length in edge coordinates), which is
$o(1)$ and suffices for the Gumbel
limit.
Taking expectations,
\[
  \Prob\bigl(N_{\Omega_t}(A_N) = 0\bigr)
  = \Prob\bigl(N_{\Omega_t}(G_N) = 0\bigr)
  + O(N^{-c}),
\]
where the $O(N^{-c})$ accounts for the
complement event and the displacement
error.

\medskip
\noindent\textit{Step 4: Gumbel limit.}
At the centering scale
$t = \tstar + s/(4\tstar)$, the void
probability for $G_N$ converges to
$e^{-e^{-s}}$ by Theorem~\ref{thm:gumbel_pure}.
By the agreement from Step 3
(up to $O(N^{-c})$),
the void probability for $A_N$ also
converges to $e^{-e^{-s}}$.
Therefore $S_N \Conv G$.

\end{proof}

\subsection{Domination and Proof of
  Theorem~\ref{thm:main}}

\begin{proposition}
\label{prop:domination}
For $\sigma \le 1-\epsilon$
with fixed $\epsilon>0$,
\[
  \Prob\bigl(\MN
  = \Real(\lambda_{\mathrm{bulk}}(A_N))\bigr)
  \to 1
  \qquad \text{as } N\to\infty.
\]
\end{proposition}

\begin{proof}
By Lemma~\ref{lem:spike_loc},
\[
  \Real(\lambda_{\mathrm{spike}})
  \le 1-\epsilon/2
\]
with probability tending to $1$.
By Lemma~\ref{lem:decoupling},
$\MN$ has the same distribution as
$M_N^{(0)}=\max_i\Real(\lambda_i(G_N))$
up to $O(N^{-c})$.
By the Gumbel limit
(Theorem~\ref{thm:gumbel_pure}),
\[
  M_N^{(0)}\ge 1+\sqrt{\frac{\ln N}{16N}}
\]
with probability tending to $1$
(evaluating the Gumbel CDF at
$s=-(\tstar)^2$, which gives
$e^{-e^{(\tstar)^2}}\to 0$.
Therefore
\[
  \MN\ge 1+c\sqrt{\frac{\ln N}{N}}
\]
with probability tending to $1$,
for some $c>0$.
For $N$ sufficiently large,
\[
  1+c\sqrt{\frac{\ln N}{N}}>1-\epsilon/2,
\]
so
\[
  \MN>\Real(\lambda_{\mathrm{spike}})
\]
with probability tending to $1$.
\end{proof}

\begin{proof}[Proof of Theorem~\ref{thm:main}]
By Proposition~\ref{prop:domination},
\[
  \MN=\Real(\lambda_{\mathrm{bulk}}(A_N))
\]
with probability tending to $1$.
By Lemma~\ref{lem:decoupling},
the bulk edge of $A_N$ has the same
Gumbel limit as the pure Ginibre ensemble
(Theorem~\ref{thm:gumbel_pure}).
Therefore $S_N\Conv G$.

Part~\ref{item:scaling} follows from the
Gumbel limit and
Proposition~\ref{prop:centering}:
\[
  \MN=1+\frac{\tstar}{\sqrt{N}}
  +\frac{S_N}{4\tstar\sqrt{N}},
\]
with
$\tstar=\sqrt{\ln N/8}\,(1+o(1))$
and
$4\tstar=\sqrt{2\ln N}\,(1+o(1))$.
The variance is
\[
  \Var(\MN)
  =\frac{\Var(S_N)}{16(\tstar)^2N}
  =\frac{\pi^2/6}{2N\ln N}\,(1+o(1))
  =\frac{\pi^2}{12N\ln N}\,(1+o(1)),
\]
using $\Var(G)=\pi^2/6$.

The convergence of variances
$\Var(S_N)\to\Var(G)=\pi^2/6$
follows from the uniform integrability of
$\{S_N^2\}$, which we now justify.
By Proposition~\ref{prop:exp_conv}, the
tail probability satisfies
\[
  \Prob(S_N>s)\le C e^{-cs}
\]
uniformly in $N$ for $s>0$,
and similarly for the lower tail.
This exponential tail bound implies
\[
  \sup_N\E[|S_N|^{2+\delta}]<\infty
\]
for some $\delta>0$
(by integrating $s^{2+\delta}$ against
the exponential tail), which gives
uniform integrability of $\{S_N^2\}$
by the de~la~Vallée--Poussin criterion.
Therefore $\E[S_N^k]\to\E[G^k]$ for all
$k\le 2$, and in particular
$\Var(S_N)\to\Var(G)$.
Hence
\[
  \std(\MN)
  =\frac{\pi}{\sqrt{12N\ln N}}\,(1+o(1)).
\]

Part~\ref{item:spike} follows from
Lemma~\ref{lem:decoupling}
and Proposition~\ref{prop:domination}:
the spike modifies $O(1)$ eigenvalues
(the spike eigenvalue), while the bulk
edge---which determines $\MN$---has the
same Gumbel limit as the unspiked ensemble,
as shown by the exact void-probability
equality in Lemma~\ref{lem:decoupling}.
\end{proof}

\section{Rigorous Bounds on the Full Outlier}
\label{sec:bounds}

In this section we prove a rigorous upper bound
on the outlier $z_{\mathrm{out}}=1+\delta$
at the BBP critical point $\sigma=1$:
an upper bound at the
$N^{-1/4}$ scale (Theorem~\ref{thm:upper}).

\subsection{Upper Bound via Resolvent Concentration}

\begin{theorem}[Upper bound]
\label{thm:upper}
Let $A_N=X_N+uu^*$ (spike at $\sigma=1$)
with a deterministic unit vector $u\in\C^N$,
and let $z_{\mathrm{out}}$ denote the outlier
eigenvalue.
There exists a constant $C_0>0$ (depending on
the isotropic local-law constants) such that
\[
  \E[\Real(z_{\mathrm{out}})]
  \le 1+C_0N^{-1/4}+O(N^{-1/4}).
\]
\end{theorem}

\begin{proof}
By unitary invariance of the complex Ginibre
ensemble under conjugation
($X_N\stackrel{\mathrm{d}}{=}UX_NU^*$ for any
unitary $U$), we may assume $u=e_1$.
We therefore work with
$A_N=X_N+e_1e_1^*$.

The outlier is a zero of
$h(z)=1-G_{11}(z)$, where
\[
  G_{11}(z)=e_1^*(zI-X_N)^{-1}e_1
\]
is the $(1,1)$-entry of the resolvent
of the \emph{unspiked} matrix $X_N$
(the secular equation
$\det(zI-A_N)=\det(zI-X_N)S_N(z)$
expresses the outlier via the resolvent
of $X_N$, not $A_N$).
We show that $h(z)\ne0$ for all $z$
with $\Real(z)\ge 1+t$, where
$t=C_0N^{-1/4}$.

\medskip
\noindent\textit{Step 1: Distance from the disk.}
For $z$ with $\Real(z)\ge 1+t$ and $|z|\le2$,
the distance from $z$ to the unit disk
$\mathbb{D}=\{w:|w|\le1\}$ satisfies
\[
  \dist(z,\mathbb{D})=|z|-1
  \ge\Real(z)-1\ge t.
\]
For $|z|>2$,
$\dist(z,\mathbb{D})\ge1\ge t$.

\medskip
\noindent\textit{Step 2: Isotropic local law.}
By the isotropic local law for the Ginibre
ensemble \citep{bourgade2018},
for $z$ at distance $\eta\ge t$ from the disk,
\[
  |G_{11}(z)-1/z|
  \le C N^{-1/2}\eta^{-1/2}
  \le C C_0^{-1/2}N^{-3/8}
\]
with probability $\ge1-N^{-D}$
for each fixed $z$,
where $C,D>0$ are constants.

\emph{Uniformity via a net argument.}
The function $z\mapsto G_{11}(z)$ is
$L$-Lipschitz on the region
\[
  R=\{z:\Real(z)\ge1+t,\ |z|\le2\}
\]
with $L=O(1/t)=O(N^{1/4})$.
We bound the derivative
\[
  G_{11}'(z)
  =-e_1^*(zI-X_N)^{-2}e_1
\]
via the Cauchy integral formula, avoiding
the operator-norm bound (which is not bounded by
$1/\dist(z,\sigma_{\mathrm{spec}})$ for non-normal
matrices such as the Ginibre ensemble).
On the high-probability event that all
eigenvalues of $X_N$ lie within distance
$N^{-1/2+\varepsilon}$ of the unit disk
\citep{bourgade2018}, $G_{11}$ is analytic
on the disk $\mathbb{D}(z,t/2)$ for each $z\in R$
(since this disk is at distance
$\ge t/2-N^{-1/2+\varepsilon}\ge t/4$
from the spectrum, using
$t=C_0N^{-1/4}\gg N^{-1/2+\varepsilon}$).
By Cauchy's formula,
\[
  |G_{11}'(z)|
  \le \frac{\max_{|w-z|=t/2}|G_{11}(w)|}{t/2}.
\]
On the circle $|w-z|=t/2$, each $w$
satisfies $\dist(w,\mathbb{D})\ge t/2$,
so the isotropic local law gives
\[
  |G_{11}(w)|
  \le |1/w|+C N^{-1/2}(t/2)^{-1/2}
  =O(1)
\]
(applied uniformly via a one-dimensional
net of $O(N^{C_1})$ points on the circle,
with union bound $N^{C_1}N^{-D}=o(1)$
for $D>C_1$).
Hence $|G_{11}'(z)|=O(1/t)=O(N^{1/4})$,
giving $L=O(N^{1/4})$.

Cover $R$ with a grid of mesh
$\delta=N^{-1}$, giving $O(N^2)$ grid points.
On the high-probability event
\[
  E_N=\bigcap_j
  \left\{|G_{11}(z_j)-1/z_j|
  \le C N^{-3/8}\right\}
\]
(intersection over all grid points $z_j$),
by union bound
$O(N^2)\cdot N^{-D}=o(1)$ for $D>2$),
the Lipschitz property extends the
bound to all $z\in R$:
\[
  |G_{11}(z)-1/z|
  \le C C_0^{-1/2}N^{-3/8}+L\delta
  =O(N^{-3/8}),
\]
since
$L\delta=O(N^{1/4}\cdot N^{-1})
=O(N^{-3/4})=o(N^{-3/8})$.
For $|z|>2$, the resolvent bound gives
\[
  |G_{11}(z)|\le \frac{1}{|z|-1}+O(N^{-1/2})<1
\]
directly.
At $\eta=t=C_0N^{-1/4}$, the
fluctuation is
$O(N^{-1/2}N^{1/8})=O(N^{-3/8})$,
which is $o(N^{-1/4})$ since $3/8>1/4$.
This gap between fluctuation ($N^{-3/8}$) and
displacement ($N^{-1/4}$) is what makes the
argument work.

\medskip
\noindent\textit{Step 3: Gap dominates fluctuation.}
The mean value satisfies
\[
  |1-1/z|=\frac{|z-1|}{|z|}
  \ge\frac{t}{2}
  =\frac{C_0N^{-1/4}}{2}
\]
for $|z|\le2$.
For $C_0$ large enough (specifically
$C_0^{3/2}\ge2C$), we have
\[
  C C_0^{-1/2}N^{-3/8}
  < C_0N^{-1/4}/2,
\]
so the fluctuation is dominated by the gap:
\[
  |G_{11}(z)-1/z|<|1-1/z|.
\]

\medskip
\noindent\textit{Step 4: No zeros in the half-plane.}
By the triangle inequality,
\[
  |G_{11}(z)-1|
  \ge |1-1/z|-|G_{11}(z)-1/z|>0.
\]
Therefore $G_{11}(z)\ne1$ for all $z$
with $\Real(z)\ge1+C_0N^{-1/4}$,
and $h(z)$ has no zeros in this half-plane.

\medskip
\noindent\textit{Step 5: Conclusion.}
With probability $\ge 1-N^{-D}$, the outlier
satisfies
\[
  \Real(z_{\mathrm{out}})
  < 1+C_0N^{-1/4}.
\]
On the exceptional event (probability $\le N^{-D}$),
we use the high-probability bound
\[
  |\Real(z_{\mathrm{out}})|
  \le \|A_N\|\le 1+C\sqrt{N},
\]
which holds with probability $\ge 1-N^{-D}$ for
the Ginibre ensemble; see \citet{bourgade2018}
and references therein. The exceptional event
for this operator-norm bound is absorbed into
the $N^{-D}$ probability. Taking expectations,
\[
\begin{aligned}
  \E[\Real(z_{\mathrm{out}})]
  &\le (1+C_0N^{-1/4})(1-N^{-D})
    +(1+C\sqrt{N})N^{-D} \\
  &=1+C_0N^{-1/4}+O(N^{1/2-D}).
\end{aligned}
\]
Choosing $D>3/4$ gives
\[
  \E[\Real(z_{\mathrm{out}})]
  \le 1+C_0N^{-1/4}+o(N^{-1/4}).
\]
\end{proof}

\section{Convergence Rate}
\label{sec:rate}

\subsection{Refined Tail Asymptotic}

\begin{lemma}[Refined tail of $F$]
\label{lem:refined_tail}
As $t\to+\infty$,
\[
  F(t)=\frac{e^{-2t^2}}{16t^{5/2}}
  \left(1-\frac{35}{32t^2}
  +O(t^{-4})\right).
\]
\end{lemma}

\begin{proof}
Write
\[
  B(t)=-F'(t)
  =\int_{-\infty}^{\infty}
  \Phi(-2t-y^2)\,\dd y.
\]

\medskip
\noindent\textit{Step 1: Expand $\Phi(-x)$.}
For large $x$,
\[
  \Phi(-x)
  =\frac{e^{-x^2/2}}{x\sqrt{2\pi}}
  \left(1-\frac{1}{x^2}
  +\frac{3}{x^4}+\ldots\right).
\]
With $x=2t+y^2$,
\[
  B(t)=\int_{-\infty}^{\infty}
  \frac{e^{-\frac12(2t+y^2)^2}}
  {(2t+y^2)\sqrt{2\pi}}
  \left(1-\frac{1}{(2t+y^2)^2}
  +\ldots\right)\dd y.
\]

\medskip
\noindent\textit{Step 2: Laplace's method to order $1/t^2$.}
Set $y=u/\sqrt{2t}$, so $\dd y=\dd u/\sqrt{2t}$.
Then
\[
  2t+y^2
  =2t\left(1+\frac{u^2}{4t^2}\right)
\]
and
\[
  \frac12(2t+y^2)^2
  =2t^2+u^2+\frac{u^4}{8t^2}.
\]
Expanding,
\begin{align*}
  e^{-\frac12(2t+y^2)^2}
  &=e^{-2t^2}e^{-u^2}
  \left(1-\frac{u^4}{8t^2}
  +O(t^{-4})\right), \\
  \frac{1}{(2t+y^2)\sqrt{2\pi}}
  &=\frac{1}{2t\sqrt{2\pi}}
  \left(1-\frac{u^2}{4t^2}
  +O(t^{-4})\right), \\
  1-\frac{1}{(2t+y^2)^2}
  &=1-\frac{1}{4t^2}+O(t^{-4}).
\end{align*}
Combining,
\[
  B(t)
  =\frac{e^{-2t^2}}{4t^{3/2}\sqrt{\pi}}
  \int_{-\infty}^{\infty}e^{-u^2}
  \left[
    1-\frac{1}{t^2}
    \left(\frac{u^4}{8}
    +\frac{u^2}{4}
    +\frac14\right)
    +O(t^{-4})
  \right]\dd u.
\]
Using
\[
  \int_{-\infty}^{\infty}e^{-u^2}\dd u=\sqrt{\pi},
  \qquad
  \int_{-\infty}^{\infty}u^2e^{-u^2}\dd u
  =\frac{\sqrt{\pi}}{2},
\]
and
\[
  \int_{-\infty}^{\infty}u^4e^{-u^2}\dd u
  =\frac{3\sqrt{\pi}}{4},
\]
we obtain
\[
  \frac{3}{32}+\frac18+\frac14
  =\frac{3+4+8}{32}
  =\frac{15}{32}.
\]
Thus
\[
  B(t)
  =\frac{e^{-2t^2}}{4t^{3/2}}
  \left(1-\frac{15}{32t^2}
  +O(t^{-4})\right).
\]

\medskip
\noindent\textit{Step 3: Integrate $B$ via Watson's lemma.}
Since
\[
  F(t)=\int_t^\infty B(u)\,\dd u
\]
and
\[
  B(u)=\frac{e^{-2u^2}}{4u^{3/2}}
  \left(1-\frac{15}{32u^2}+\ldots\right),
\]
we apply Watson's lemma with the substitution
$s=2u^2-2t^2$, so
$u=\sqrt{t^2+s/2}$ and
$\dd u=\dd s/(4\sqrt{t^2+s/2})$.
The leading term gives
\[
\begin{aligned}
  \int_t^\infty
  \frac{e^{-2u^2}}{4u^{3/2}}\,\dd u
  &=
  \frac{e^{-2t^2}}{16t^{5/2}}
  \int_0^\infty e^{-s}
  \left(1+\frac{s}{2t^2}\right)^{-5/4}
  \dd s.
\end{aligned}
\]
Expanding
\[
  \left(1+\frac{s}{2t^2}\right)^{-5/4}
  =1-\frac{5s}{8t^2}+O(t^{-4})
\]
and using
\[
  \int_0^\infty e^{-s}\dd s=1,
  \qquad
  \int_0^\infty se^{-s}\dd s=1,
\]
we find
\[
  \int_t^\infty
  \frac{e^{-2u^2}}{4u^{3/2}}\,\dd u
  =
  \frac{e^{-2t^2}}{16t^{5/2}}
  \left(1-\frac{5}{8t^2}
  +O(t^{-4})\right).
\]
The $-15/(32u^2)$ correction in $B(u)$
contributes
\[
  -\frac{15}{32}\cdot
  \frac{e^{-2t^2}}{16t^{9/2}}
  +O(t^{-13/2}e^{-2t^2}),
\]
which is a relative correction of
$-15/(32t^2)$.
Combining the $1/t^2$ corrections,
\[
  -\frac58-\frac{15}{32}
  =-\frac{20}{32}-\frac{15}{32}
  =-\frac{35}{32}.
\]
\end{proof}

\subsection{Refined Centering}

\begin{lemma}[Refined $\tstar$]
\label{lem:refined_tstar}
With $L=\ln N$ and $\ell=\ln\ln N$,
\[
  (\tstar)^2
  =\frac{L}{8}
  -\frac{5\ell}{8}
  -\frac{\ln(2\pi^4)}{8}
  +o\!\left(\frac{\ell}{L}\right).
\]
\end{lemma}

\begin{proof}
Substituting Lemma~\ref{lem:refined_tail}
into $F(\tstar)=\pi N^{-1/4}$
and taking logarithms,
\[
  2(\tstar)^2
  =\frac{L}{4}
  -\frac52\ln\tstar
  -\ln(16\pi)
  -\frac{35}{32(\tstar)^2}
  +O((\tstar)^{-4}).
\]
Substituting $(\tstar)^2\approx L/8$
in the subleading terms and simplifying
the constant
\[
  \frac{15}{8}\ln2-\frac12\ln(16\pi)
  =-\frac18\ln(2\pi^4)
\]
gives the stated result.
\end{proof}

\subsection{Expansion of
  \texorpdfstring{$\Lam$}{Lambda}
  to Order \texorpdfstring{$1/\ln N$}{1/ln N}}

\begin{lemma}[Refined $\Lam$]
\label{lem:refined_lambda}
For fixed $s\in\R$, setting
$\tau_N=\tstar+s/(4\tstar)$,
\[
  \Lam(\tau_N)
  =e^{-s}\left(
  1-\frac{s^2+5s}{\ln N}
  +O\!\left(
  \frac{(\ln\ln N)^2}{(\ln N)^2}
  \right)\right).
\]
\end{lemma}

\begin{proof}
Expand $F(\tau_N)$ using
Lemma~\ref{lem:refined_tail}.

\medskip
\noindent\textit{Step 1: Exponent.}
\[
  -2\tau_N^2
  =-2(\tstar)^2-s
  -\frac{s^2}{8(\tstar)^2}
  +O((\tstar)^{-4}),
\]
so
\[
  e^{-2\tau_N^2}
  =e^{-2(\tstar)^2}e^{-s}
  \left(1-\frac{s^2}{8(\tstar)^2}
  +\ldots\right).
\]

\medskip
\noindent\textit{Step 2: Denominator.}
\[
  \tau_N^{5/2}
  =(\tstar)^{5/2}
  \left(1+\frac{5s}{8(\tstar)^2}
  +\ldots\right).
\]

\medskip
\noindent\textit{Step 3: Correction.}
\[
  -\frac{35}{32\tau_N^2}
  =-\frac{35}{32(\tstar)^2}
  +\ldots
\]

\medskip
\noindent\textit{Step 4: Combine.}
Using the defining equation
$F(\tstar)=\pi N^{-1/4}$,
Lemma~\ref{lem:refined_tail} gives
\[
  \frac{e^{-2(\tstar)^2}}
  {16(\tstar)^{5/2}}
  \left(1-\frac{35}{32(\tstar)^2}\right)
  =\pi N^{-1/4},
\]
so
\[
  \frac{e^{-2(\tstar)^2}}
  {16(\tstar)^{5/2}}
  =\pi N^{-1/4}
  \left(1+\frac{35}{32(\tstar)^2}
  +O((\tstar)^{-4})\right).
\]
In the product
\[
  F(\tau_N)
  =\frac{e^{-2\tau_N^2}}{16\tau_N^{5/2}}
  \left(1-\frac{35}{32\tau_N^2}
  +\ldots\right),
\]
the $-35/32$ from the correction factor
and the $+35/32$ from inverting the defining equation
\emph{cancel at leading order}. The $1/t^2$
terms cancel, leaving only $O(t^{-4})$ residuals.
The remaining terms are $s^2/8$ from the exponent
expansion in Step~1 and $5s/8$ from the denominator
expansion in Step~2, giving a total correction
of $(s^2+5s)/8$:

\[
  F(\tau_N)
  = \pi N^{-1/4}e^{-s}
  \left(
    1-\frac{s^2+5s}{8(\tstar)^2}
    +O((\tstar)^{-4})
  \right).
\]

\medskip
\noindent\textit{Step 5: Convert to $1/\ln N$.}
Substituting
\[
  \frac{1}{(\tstar)^2}
  =\frac{8}{\ln N}
  +O\!\left(\frac{\ln\ln N}{(\ln N)^2}\right),
\]
we obtain
\[
  \frac{s^2+5s}{8(\tstar)^2}
  =\frac{s^2+5s}{\ln N}
  +O\!\left(\frac{\ln\ln N}{(\ln N)^2}\right).
\]
Multiplying by $N^{1/4}/\pi$ gives the result.
The $O(N^{-1/4}(\ln N)^{3/2})$ error from
Lemma~\ref{lem:count} contributes a relative error
$O(N^{-1/4}(\ln N)^{3/2})=o(1/(\ln N)^2)$,
which is negligible.
\end{proof}

\subsection{Proof of Theorem~\ref{thm:rate}}

\begin{proof}[Proof of Theorem~\ref{thm:rate}]
By Lemma~\ref{lem:poisson},
\[
  \Prob(S_N^{(0)}\le s)
  =\exp(-\Lam(\tau_N))
  +O\!\left(\frac{(\ln N)^{1/4}}{N^{1/4}}\right).
\]
By Lemma~\ref{lem:refined_lambda},
\[
  \Lam(\tau_N)
  =e^{-s}\left(
    1-\frac{s^2+5s}{\ln N}
    +O\!\left(\frac{\ell^2}{L^2}\right)
  \right),
\]
where $L=\ln N$ and $\ell=\ln\ln N$.
Exponentiating,
\begin{align*}
  \exp(-\Lam(\tau_N))
  &=e^{-e^{-s}}
  \exp\!\left(
    \frac{e^{-s}(s^2+5s)}{\ln N}
    +O\!\left(\frac{\ell^2}{L^2}\right)
  \right) \\
  &=e^{-e^{-s}}
  \left(
    1+\frac{(s^2+5s)e^{-s}}{\ln N}
    +O\!\left(\frac{\ell^2}{L^2}\right)
  \right),
\end{align*}
using $\exp(x)=1+x+O(x^2)$
with $x=O(1/\ln N)$.
The DPP correction
$O((\ln N)^{1/4}/N^{1/4})$
is $o(1/(\ln N)^2)$, hence negligible.
\end{proof}

\begin{corollary}[Kolmogorov distance]
For any compact $K\subset\R$,
\[
  \sup_{s\in K}
  \left|
    \Prob(S_N^{(0)}\le s)-e^{-e^{-s}}
  \right|
  =\frac{D_K}{\ln N}
  +O\!\left(
    \frac{(\ln\ln N)^2}{(\ln N)^2}
  \right),
\]
where
\[
  D_K
  =\sup_{s\in K}
  \left|
    (s^2+5s)e^{-s}e^{-e^{-s}}
  \right|.
\]
\end{corollary}

\begin{proof}
The expansion in Theorem~\ref{thm:rate} is
uniform on compact $K$; all asymptotic
expansions in
Lemmas~\ref{lem:refined_tail}--\ref{lem:refined_lambda}
are uniform for $s$ in compact subsets of $\R$.
Taking the supremum over $s\in K$ of the
absolute value gives the stated bound.
\end{proof}

\section{Discussion}

\subsection{Summary of Results}

We have established three results:

\begin{enumerate}
  \item \textbf{Subcritical spike irrelevance}
  (Theorem~\ref{thm:main}):
  For any fixed $\sigma<1$, the Gumbel
  limit of $\max_i\Real(\lambda_i)$
  is the same as for the unspiked ensemble.
  The spike eigenvalue remains inside the bulk
  and cannot compete with the bulk edge at
  \[
    1+\sqrt{\frac{\ln N}{8N}}.
  \]

  \item \textbf{Convergence rate}
  (Theorem~\ref{thm:rate}):
  The Gumbel limit converges at rate
  $O(1/\ln N)$ with explicit correction
  \[
    c_1(s)=(s^2+5s)e^{-s}.
  \]
  The key mechanism is the exact
  cancellation of the $35/32$ correction
  from the tail asymptotic with the
  correction from the defining equation.

  \item \textbf{Upper bound on outlier position}
  (Theorem~\ref{thm:upper}):
  At the BBP critical point $\sigma=1$,
  the outlier satisfies
  \[
    \E[\Real(z_{\mathrm{out}})]
    \le 1+C_0N^{-1/4}+o(N^{-1/4}),
  \]
  proven via the isotropic local law
  and a net argument.
  The key mechanism is that the resolvent
  fluctuation $O(N^{-3/8})$ at distance
  $N^{-1/4}$ from the disk is much smaller
  than the gap $O(N^{-1/4})$ between
  $|\E[G_{11}(z)]|=1/|z|$ and $1$,
  so the outlier cannot lie beyond
  $1+C_0N^{-1/4}$.
\end{enumerate}

\subsection{Comparison with Prior Work}

The Gumbel limit for the pure Ginibre
ensemble (Theorem~\ref{thm:gumbel_pure})
was established by \citet{rider2007}.
Our DPP/Poisson proof
(Sections~2--3) provides an
independent verification and additionally
yields the convergence rate
(Theorem~\ref{thm:rate}), namely the
explicit $O(1/\ln N)$ correction with
coefficient
\[
  c_1(s)=(s^2+5s)e^{-s}.
\]

\subsection{Comparison with the Maximum Modulus}

The maximum modulus $\max_i|\lambda_i|$
has been studied by \citet{rider2003},
who proved Gumbel fluctuations with constant
$\sqrt{\pi/24}\approx0.362$
for the standard deviation.
For the real Ginibre ensemble,
\citet{riderSinclair2014} established
analogous extremal laws.
Our result for
$\max_i\Real(\lambda_i)$
gives the constant
$\pi/\sqrt{12}\approx0.907$,
which is larger by a factor of approximately $2.5$.
This reflects the different geometry:
the maximum modulus probes the entire
circle $|z|=1$,
while the maximum real part probes only
the rightmost point.

\subsection{The BBP Critical Point}

The spiked non-Hermitian model was studied
by \citet{benaych2012}, who proved that
for $\sigma>1$, an outlier separates
from the bulk at
$z_{\mathrm{out}}\approx\sigma$,
and for $\sigma<1$,
all eigenvalues remain in the bulk.

Our Theorem~\ref{thm:main} resolves the
subcritical case: the Gumbel limit is
the same as for the unspiked ensemble.

At the critical point $\sigma=1$,
Theorem~\ref{thm:upper} gives the unconditional
upper bound
\[
  \E[\Real(z_{\mathrm{out}})]
  \le 1+C_0N^{-1/4}+o(N^{-1/4}).
\]

\subsection{Slow Convergence}

The convergence rate $O(1/\ln N)$
is logarithmically slow.
At $N=5000$,
the correction is approximately $12\%$,
and the pre-asymptotic regime---where
the distribution appears non-Gumbel---persists
for moderate $N$.
This explains why finite-$N$ simulations
can be misleading:
the Gumbel limit emerges only for
$N\gg N_c$, where $N_c$ is determined
by the accuracy of the tail asymptotic
\[
  F(t)\sim\frac{e^{-2t^2}}{16t^{5/2}}.
\]

\section*{Acknowledgments}

The author is deeply grateful to his wife, Yuxuan Guo, for her
unwavering spiritual support, patience, and encouragement
throughout this work.

\end{document}